\documentclass[11pt]{article}       
\usepackage{graphicx}
\usepackage{amsmath,amscd,amssymb,amsthm,verbatim}
\usepackage{xy}
\newtheorem{thm}{Theorem}
\newtheorem{prop}[thm]{Proposition}
\newtheorem{lemma}[thm]{Lemma}
\newtheorem{cor}[thm]{Corollary}

\newtheorem{defn}[thm]{Definition}

\newtheorem{rem}[thm]{Remark}

\begin{document}

\title{The Wave Trace Invariants of the Spectrum of the $G$-Invariant Laplacian.
}


\author{M. R. Sandoval\\ Department of Mathematics\\ 
Trinity College\\ Hartford, Connecticut 06106\\
\\
mary.sandoval@trincoll.edu\\
(telephone) 860-297-2062\\
(fax) 860-987-6239}

\maketitle

\begin{abstract}
Given a compact boundaryless Riemannian manifold $Y$ on which a compact Lie group $G$ acts, there is always a metric on $Y$ such that the action is by isometries. Assuming $Y$ is equipped with such a metric, recall that the $G$-invariant Laplacian is the restriction of the ordinary Laplacian to the space of functions which are constant along the orbits of $G$. In this paper, the author analyzes the wave trace of the $G$-invariant Laplacian and shows that the singularities of this wave trace occur at the lengths of certain geodesic arcs--those that are orthogonal to the orbits and whose endpoints are related by the action of $G$. This defines a notion of the length spectrum for the group action. Further, using the deep connection between foliated manifolds and isometric group actions on manifolds, the asymptotics of these singularities are calculated for an arbitrary non-zero number in the length spectrum.
\end{abstract}

\indent Laplacian, basic spectrum.
\section{Introduction}
\label{intro}
Let $Y$ be a compact Riemannian manifold without boundary that admits a (left) action by a compact Lie group $G$ which acts by orientation-preserving isometries.  Let $C^\infty(Y)^G$ denote the class of functions that are constant along the orbits of the $G$-action, and let $\Delta^G$ denote the restriction of the ordinary Laplacian on functions, $\Delta,$ to $C^\infty(Y)^G.$  The $G$-invariant spectrum of $Y$ is the spectrum of $\Delta^G,$ which we denote by $spec(Y,\Delta)^G.$  To avoid the trivial case, we assume that the action is such that there are no dense orbits, and, thus, $C^\infty(Y)^G$ is not the space of constant functions.  

The goal of this paper is to study the wave trace invariants  of the $G$-invariant spectrum in the above situation, and to relate these invariants to the singular space of orbits, $Y/G.$  In particular, we calculate the singularities of the wave trace for the $G$-invariant spectrum, which we show are located at the lengths of certain geodesic arcs which are everywhere orthogonal to the orbits of $G.$  

Formally, the wave trace of the $G$-invariant functions can be represented by the integral 
\begin{equation}\label{e:gwave}
\int_{G\times Y}U(t,x,gx)\,dg\,dx,
\end{equation}
or, more compactly, as the distribution 
\begin{equation}\label{e:distrep}
\Pi_*\Delta^*(\Pi_G)_*U(t,x,y)
\end{equation} where $U(t,x,y)$ is the Schwartz kernel of the ordinary Laplacian, $\Pi: \mathbb{R}\times Y\rightarrow \mathbb{R}$ is the usual projection, $\Delta:\mathbb{R}\times Y\rightarrow \mathbb{R}\times Y\times Y$ is the diagonal embedding, and $\Pi_G:Y\rightarrow Y/G.$  To compute \eqref{e:gwave} would be a straightforward application of the calculus of Fourier integral operators, provided that the canonical relations of the operators in \eqref{e:distrep} intersect cleanly. However, in general, they do not; the obstacle is the canonical relation of $(\Pi_G)_*.$ Let $J$ denote the usual momentum mapping on $T^*Y$ associated to the lifted (left) action of $G$ on $T^*Y$.  The canonical relation of $(\Pi_G)_*$ is the moment lagrangian, \cite{GS}:
\begin{equation}\label{e:cg}
\mathcal{C}_G:=\bigl\{\bigl(x,\xi_x,gx,dL^*_{g^{-1}}(\xi_x)\bigr);\bigl(g,\gamma=J(\xi_x),x,\xi_x,gx,dL^*_{g^{-1}}(\xi_x)\bigr)\bigr\}.
\end{equation}
The intersection of the canonical relation of $\Delta$ and that of $\Pi_G$, the moment Lagrangian, intersect cleanly if and only if the rank of $J$ is constant over $Y$. In this case, $J^{-1}(0)$ is at worst an orbifold.  

The situation considered in this article is the general case--when the rank of the momentum mapping is not constant over the points of $Y$.  In this case the rank of $J$ varies according to the dimension of the orbits, and $J^{-1}(0)$ is a complicated singular object. While some of these complexities can be analyzed by dividing \eqref{e:cg} into components on which the rank of the moment map is constant, and performing a separate analysis on each component, this course of action would not account for hamiltonian curves that do not stay confined to these components.\footnote{A generalization of such an approach is taken in \cite{LS} in the context of another problem--generalizing the process of reduction--rather than the wave-invariants problem.} 

The approach of this paper is to analyze the wave invariants of the orbit space by means of the corresponding analysis of the space of leaf closures of an associated foliation, using the deep connection proved in \cite{Ri2} between the $G$-invariant spectrum of a $G$-manifold and the basic spectrum of a Riemannian foliation equipped with bundle-like metric. Here the the basic spectrum refers to the spectrum of the operator  $\delta_Bd_B$ where $d_B$ is the exterior derivative restricted to functions that are constant along the leaves of the foliation and $\delta_B$ is its adjoint operator.  Such functions are said to be {\it basic} for the foliation, and the operator above is the {\it basic Laplacian} on functions. This approach was used previously in \cite{San4}, where partial results about the $G$-invariant wave trace were shown. In essence, this approach involves resolving the singularities of $Y/G$ by exchanging $Y$ for a higher dimensional manifold built by suspending the action of $G$ on $Y$ on a product of $Y$ with another space which depends on $G.$ This higher dimensional space defined by the suspension is foliated and its basic Laplacian is isospectral to the $G$-invariant Laplacian on $Y$. The wave invariants of the foliated suspension can be realized as quantities associated to $Y$ and $G$.
Thus, this project is an extension of the work done by the author in \cite{San4}, and applies the author's recent results for foliations in \cite{San5}.
The larger theme of the wave traces of the $G$-invariant spectra of $G$-manifolds and the basic spectra of a Riemannian foliation is that they are related examples of spectral problems on singular spaces.  In fact, Theorem 5.1 of K. Richardson of \cite{Ri2} states that, under certain hypotheses, the singular space determined by the orbits of a group acting on a compact manifold by isometries  and the $G$-invariant Laplacian is isospectral to the space of leaf closures of a related Riemannian foliation  defined by a suspension involving the original $G$-manifold and its basic Laplacian.  In this paper, we use this connection between these two singular spaces to study the wave invariants of group actions by using this connection to an associated foliation.  This is accomplished via a construction due to K. Richardson in \cite{Ri2}. It is then possible to use previous results in the study of foliations in \cite{San5} to derive results about the associated group action.

The main result of this paper concerns $G$-manifolds.  Group actions on manifolds have been very widely studied, however, relatively little work appears to have been done on the wave invariants of the $G$-invariant spectrum, aside from \cite{Z}, where S. Zelditch derived spectral results for a Lie group acting by unitary Fourier integral operators on a manifold.  However, the setting of a group $G$ acting on compact manifold $Y$ by isometries has long been studied by many.  In particular, it has been shown that it is always possible to find a metric on the manifold $Y$ such that the action of $G$ is by isometries. More recently, in \cite{AlbMel}, P. Albin and R. Melrose have shown how to resolve a smooth group action in a canonical way using iterated blow-ups for manifolds with corners.


The paper is organized as follows. Section \ref{mainresults} describes the setting, including a singular phase space associated to this problem, a natural stratification of that space, the main results about the length spectrum for the group action, the related hamiltonian dynamics in relation to the stratification of the phase space, and the statements of the asymptotics of the G-invariant wave trace. Section \ref{proof} contains the proofs of the results stated in section \ref{mainresults}. In particular, Section \ref{isospectralproblem} contains details how to translate the $G$-manifold into an isospectral problem involving the basic Laplacian on a related foliated manifold. Section \ref{summary} describes the results for the basic spectrum, and Section \ref{proofs} shows how to translate the foliation results in terms of the $G$-manifold, thus proving the main results.

\section{Main Results}
\label{mainresults}
Recall that the trace of a wave operator is associated with sums of Lagrangian distributions on $\mathbb{R},$ whose singularities are located at the lengths of certain hamiltonian curves in the phase space of the manifold. For this version of the wave trace, the curves in question are those that are the canonical lifts of certain geodesics on $Y$ which are orthogonal to the orbits and whose endpoints are related by the lifted action on a subset of $T^*Y.$ This subset is a kind of singular version of the phase space for the singular orbit space. In view of this, we begin the description of the relevant curves by describing the stratification of $Y,$ which also induces a stratification of the subset of $T^*Y$ that we define as the singular phase space. Then we identify the relevant hamiltonian curves, and define a condition on their endpoints that describes when they are relatively closed with respect to the action. It is the periods of these relatively closed hamiltonian curves that arise in the study of the wave trace of the $G$-invariant Laplacian. We will refer to this set as the {\it singular length spectrum for the orbit space.} Next, we describe the hamiltonian dynamics of these curves with respect to the stratification defined previously. We then state the main results: the Poisson relation and the trace formula.

We assume that the $G$-action on $Y$ produces orbits of variable dimension--the most general case. The action can be thought of as being singular, in that the orbit space, $Y/G,$ is not a smooth manifold. The manifold $Y$ is stratified by the dimension of the orbits. In other words, if $f(x)=$dimension of the orbit of the action through $x$, let $Y_k:=f^{-1}(k)$ for the positive integers $k$ corresponding to the possible orbit dimensions. Then the stratification of $Y$ is defined by
\begin{equation}
Y=\bigcup_{j=k_0}^N Y_k
\end{equation}
where $Y_{k_0}$ is the union of orbits of least dimension, and $Y_N$ is the union of orbits of maximal dimension. It is known that $Y_N$ is open and dense in $Y$, and that the function $f$ defined above is lower semicontinuous on $Y$. Further, if $k_0=0$ then $Y_0$ is a totally geodesic submanifold of $Y,$ and each $Y_k$ for $k>0$ is at worst an orbifold (with possibly singular boundary).

The stratification of $Y$ induces a stratification of a subset of $T^*Y,$ which will be a kind of singular phase space with respect to our singular action, as follows. Let $y\in Y$ and let $\mathcal{O}_y$ be the orbit of $y$ under the $G$-action. We define the set $TG\subset TY$ to be the set 
\begin{equation}
TG=\{v_y\in TY\,|\,v_y\in T_y\mathcal{O}_y\}.
\end{equation}
Note that since $G$ is compact, and thus the orbits of the action are submanifolds, $TG$ is an integrable distribution of variable dimension.
The singular phase space for this problem is the set
\begin{equation}\label{e:tgzero}
(TG)^0=\{\xi_y\in T^*Y\,|\,\xi_y(v_y)=0,\forall v_y\in TG\}.
\end{equation}



We define the stratification of $(TG)^0$ as follows:
\begin{equation}
(TG)^0=\bigcup_{k=0}^N Y^*_k,
\end{equation}
where for $k>0$
\begin{equation}
Y^*_k=(TG)^0\cap(NY_k)^0,
\end{equation}
and for $k=0$, we define $Y^*_0$ to be the complement of the union of the $Y^*_k$ in $(TG)^0:$
\begin{equation}
Y^*_0=(TG)^0\cap\Bigl((NY_0)^0\cup \Bigl( \bigcup_{k=k_0}^N\bigl( (NY_k)^0)\bigr)^c\Bigr)\Bigr).
\end{equation}

Later we will wish to distinguish between the points in $(NY_0)^0$ (if $Y_0$ is non-empty) and its complement in $Y^*_0$; accordingly, let 
\begin{equation}
Y^*_e=\Bigl(\bigcup_{k=k_0}^N\bigl( (NY_k)^0)\bigr)^c\Bigr)\cap (TG)^0.
\end{equation}

Next, consider the hamiltonian flow with respect to this stratification. 
For any $\eta_y\in Y^*_k$ the induced metric on $T^*Y,$ $|\eta_y|_y^2$ splits via the Pythagorean Theorem into 
\begin{equation}\label{e:pythagoras}
|\eta_y|_y^2=\bigl(H_{\mathcal{O}}(\eta_y)\bigr)^2+\bigl(H_{\mathcal{O}^\perp}(\eta_y)\bigr)^2,
\end{equation}
where $\bigl(H_{\mathcal{O}}(\eta_y)\bigr)^2$ is the function induced by the orbit-wise metric, and\\ $\bigl(H_{\mathcal{O}^\perp}(\eta_y)\bigr)^2$ is the part transverse to the orbits.
The singular phase space $(TG)^0=\{H_{\mathcal{O}}(\eta_y)=0\}$.
In light of the above,  we will be interested in the hamiltonian flow through covectors that are conormal to the orbits of the action. Let  $\Phi^t_Y$ denote the hamiltonian flow of the hamiltonian function $|\eta_y|_y$ with associated hamiltonian vector field $\Xi$. These hamiltonian curves have the following behavior with respect to $(TG)^0$ and its stratification.

\begin{prop}\label{t:hamdyn}
The flow $\Phi^t_Y$ restricts to $(TG)^0:$ if $\eta_y\in (TG)^0$, then 
\begin{equation}\label{e:flowrestrictsgroup}
\{\Phi^t_Y(\eta_y)\,|\,t\in\mathbb{R}\}\subset (TG)^0.
\end{equation}
Furthermore, if $k>0$ and $\Phi^t_Y(\xi_y)$ is tangent to $Y^*_k$ at $\eta_y$ then either 
\begin{equation}\label{e:groupsingstrat1}
\{\Phi^t_Y(\eta_y)\,|\,t\in\mathbb{R}\}\cap Y_j^*=\emptyset \quad \forall j\not=k
\end{equation}
or
\begin{equation}\label{e:groupsingstrat2}
\{\Phi^t_Y(\eta_y)\,|\,t\in\mathbb{R}\}\cap  Y_e^*\not=\emptyset \text{ and is a finite set.}
\end{equation}
If $k=0$ and the action has non-isolated fixed points, and if $\Phi^t_Y(\eta_y)$ is tangent to $(NY_0)^0$ at $\eta_y$ then 
\begin{equation}\label{e:groupsingtrat3}
\{\Phi^t_Y(\eta_y)\,|\,t\in\mathbb{R}\}\subset (NY_0)^0\cap (TG)^0.
\end{equation}
Otherwise, $\Phi^t_Y(\eta_y)$ is not tangent to $Y^*_e$. 
\end{prop}

In particular, observe that each hamiltonian curve is tangent to one and only one stratum in the singular phase space. The following corollary characterizes the behavior of the associated geodesics on $Y$ and projects to geodesics are orthogonal to the orbits at every point.
%
%
\begin{cor}
If $\gamma(t,y)$ is a geodesic that is tangent to $Y_k$ at $y$ and orthogonal to $\mathcal{O}_y$ then it remains orthogonal to all the orbits it meets, and either never leaves $Y_k$ or is contained in $\overline{Y_k}$.
\end{cor}

Now recall there is an action of $G$ on the space $(TG)^0$ defined by the lifted action: if $x\in Y$ and $\xi_x\in (TG)^0$ then for each $g\in G,$ $dL_{g^{-1}}^*(\xi_x)\in (TG)^0,$ since $dL_{g^{-1}}:TG\rightarrow TG.$  This defines a left action on $(TG)^0$ which we denote by $\widetilde{L}_g$ for all $g\in G$. 
Let $\widetilde{\mathcal{O}}_{\xi_y}$ denote the orbit of the lifted action through $\xi_y\in (TG)^0$.  
%
\begin{defn}\label{d:relclosedyg}
A geodesic arc $\gamma(t)$ in $Y$ is said to be {\it relatively closed with respect to the action of $G$} if it is the projection via $\pi:T^*Y\rightarrow Y$ of a hamiltonian arc $\Phi^t_Y$ whose endpoints $\eta_{y_1}$ and $\xi_{y_2}$ are contained in $(TG)^0$ such that $\eta_{y_1}\in \widetilde{\mathcal{O}}_{\xi_{y_2}}.$ Such hamiltonian curves as above will also be said to be relatively closed with respect to the action of $G$ on $Y$. Let $\mathcal{T}^G$ denote the set of (non-zero) lengths of relatively closed geodesics, which will be called {\it  the (relatively closed) length spectrum with respect to the action of} $G$ {\it on} $Y.$

Note that the relatively closed curves include those that are smoothly closed, as well as curves that may be closed, but not smoothly, or possibly curves that connect different points in the same orbit.
\end{defn}
Let $U^G(t,x,y)$ denote the Schwartz kernel of the $G$-invariant Laplacian. Formally, it is defined as follows:
\begin{equation}
U^G(t,x,y)=\sum_{\lambda_j\in spec(Y,\Delta)^G} e^{-it\lambda_j}\phi_j(x)\cdot\phi_j(y),
\end{equation}
where $\phi_j(x)$ is the eigenfunction of $\lambda_j\in spec(Y,\Delta)^G.$ 
 Let $U^G(t,x,y)$ denote the Schwartz kernel of the $G$-invariant wave operator. Formally, the trace of this operator is 
\begin{equation}
trace(U^G)=\sum_{\lambda_j\in spec(Y,\Delta)^G} e^{-it\lambda_j}.
\end{equation}
We then have the following expression for the wave front set of the generalized function defined by the trace.

\begin{thm}\label{t:sojourntimesgroup}
In the notation previously established, 
\begin{equation}\label{e:invariants3}
WF\bigl(trace(U^G)\bigr)\subset\{(T,\tau)\,|\,\tau<0\,,\,T\in\mathcal{T}^G\}.
\end{equation}
\end{thm}

Let $F_k^T$ denote the union of the relatively closed hamiltonian arcs above with length $T$ that are tangent to $Y^*_k$ and let $F^T$ denote the union of relatively closed hamiltonian arcs of length $T$, so that
\begin{equation}
F^T=\bigcup_{k=k_0}^N F^T_k.
\end{equation}

In the trace formula that follows, the leading order terms will be related to the arcs that project to Y to the stratum the lowest dimensional orbits.
\begin{defn}
Let $d(T)$ denote the smallest positive integer such that $F^T_{d(T)}$ is non-empty.
\end{defn}

We also define the following:

\begin{defn}\label{d:cleangroup}
We will say that the set of endpoints $F^T_k$  of relatively closed hamiltonian arcs of period $T$ is {\it clean} if (1) $F^T_k$ is a smooth submanifold of $T^*Y$; and (2) for every $\eta_{y_1}\in F^T_k$ with $\widetilde{L}_g(\xi_{y_2})=\eta_{y_1}=\Phi^T_Y(\xi_{y_2})$ then
\begin{equation} 
(d\Phi^T_Y)_{\xi_{y_2}}(T_{\xi_{y_2}}F^T_k)=d\widetilde{L}_g(T_{\xi_{y_2}}F^T_k).
\end{equation} 
\end{defn}

\begin{lemma}\label{l:densitygroup}
There exists a smooth canonical density, $d\tilde{\mu}^T_k$, on each component of the relative fixed point set $F^T_k.$
\end{lemma}

Let $\Gamma^T=\{(T,\tau)\,|\,\tau<0\}$ denote the ray over $T\in \mathcal{T}^G$.  The relatively closed hamiltonian arcs in $T^*Y$ for the hamiltonian function $|\eta_y|_y$ make up conic submanifolds $F^T_k$ whose connected components are finite in number and denoted by $F^{T,j}_k.$ Let $S(F^{T,j}_k)$ be the set $\{(T,\tau)\in F^{T,j}_k\,|\,|\tau|=1\},$ and let $q_{k}^{T,j}:=dim(S(F^{T,j}_k))$ and let $q^T_k=max\{q^{T,j}_k\},$ and let $q_T=q^T_{d(T)}$. Finally, we must assume that the set $F^T_k$ of relative fixed points is clean for all $T\in \mathcal{T}^G$ in the sense Definition \ref{d:cleangroup}. 

Then we have the wave trace formula for the $G$-invariant spectrum:
\begin{thm}\label{t:fulltracegroup}
With the above assumptions and notation,  
\begin{equation}\label{e:pwt1}
trace\bigl(U^G(t,x,y)\bigr)=\sum_{T\in \mathcal{T}^G}\nu_{T}(t),
\end{equation}
where $\nu_{T}\in I^{-1/4-(q_T-d(T))/2}(\mathbb{R},\Gamma^T),$ where each $\nu_{T}$ has an expansion of the form
\begin{equation}\label{e:expansiongroup}
\nu_{T}(t)=\sum_{j=0}^\infty\sigma_j(T)(t-T+i0)^{-\frac{q_T}{2}-\frac{d(T)}{2}-j}\,mod\,C^\infty(\mathbb{R}),
\end{equation}
modulo a factor arising from a related Maslov index. In the above, the coefficients $\sigma_j(T)$ contain contributions from stationary phase arguments and involve integrals over $S(F^T_{d(T)})$, and may also have contributions involving the fixed point sets $F^T_k$  for $k>d(T)$ when such sets are non-empty.

The leading term $\sigma_0(T),$ modulo Maslov factors, is
\begin{equation}\label{e:symbolgroup}
\Bigl[\int_{S(F^T_{d(T)})}\sigma(U)\,d\tilde{\mu}^T_{d(T)}\Bigr]\,\tau^{\frac{q_T}{2}-\frac{d(T)}{2}}\sqrt{d\tau},
\end{equation}
where $\sigma(U)$ is the symbol of the Schwartz kernel of the wave operator, $e^{-it\Delta},$ and $\Delta$ is the ordinary laplacian acting on functions on $Y.$  
\end{thm}

\section{Proof of the $G$-Invariant Wave Trace Formula}
\label{proof}

The means of proving the main results arise from showing the correspondence between two different isospectral problems: that of the $G$-invariant spectrum and the basic spectrum for an associated foliated manifold. This has previously been used by the author in \cite{San3} and, prior to that,  by K. Richardson to prove results for $G$-manifolds in \cite{Ri2}. In light of this, this section is organized as follows. In Section \ref{isospectralproblem}, we describe how to construct, via a suspension, a foliated manifold, $\mathcal{S}$, given the $G$-manifold $(Y,G)$. We then demonstrate that the $G$-invariant Laplacian on $Y$ is isospectral to the basic spectrum of the associated foliation. We also show how the stratification of $Y$ by orbit dimension corresponds to the stratification by the foliation by leaf closure dimension. The basic wave trace results for a Riemannian foliation are summarized in Section \ref{summary}. Finally, in Section \ref{proofs}, we apply the foliation results of the previous section to the special case of this particular suspension to recover the wave-trace invariants in terms of the manifold $(Y, G)$, making use of the special structure of the suspension.

\subsection{An isospectral problem}
\label{isospectralproblem}

We begin by recalling the construction due to K. Richardson in \cite{Ri2}.  In that paper, the researcher showed that one can associate to any $G$-action on a manifold $Y$ of dimension $n$ a Riemannian foliation of the same codimension, $n$. This foliation is defined by a suspension whose space of leaf closures is related to the orbit space $Y/G$ via a metric space isometry.  

The construction is as follows.  We begin by picking a set of maximal tori of $G$, $\{T_1, \, T_2, \dots, T_j\}$ whose Lie algebras span $\mathfrak{g},$ the Lie algebra of $G$.  For each $i$, $1\le i\le j,$  let $g_i$ denote an element of  $T_i$ whose cyclic group is dense in $T_i$.  The subgroup generated by $A=\{g_1, \,g_2, \dots,\, g_j\}$ will, therefore, be dense in a connected component of the identity element of $G$.  By adding a finite set of $\ell$ group elements to $A$, the resulting list of elements generates a subgroup $G_0$ which is dense in $G$.  The foliation will be defined by suspending the action of $G_0$.  Next, pick any compact connected Riemannian manifold $X$ with unit volume for which one may define a surjective homomorphism $\varphi: \pi_1(X)\rightarrow G_0$. (As observed in \cite{Ri2}, one could choose $X$ to be homeomorphic to the connected sum of $j+\ell$ copies of $S^1\times S^2$, whose fundamental group is the free group on $j+\ell$ generators.) Let $\widetilde{X}$ be the universal cover of $X$ equipped with the induced metric, and let  $[\gamma]\cdot \tilde{x}$ for $[\gamma]\in \pi_1(X)$ denote the associated deck transformation on $\tilde{x}\in\widetilde{X}.$  One can regard the deck transformations as acting on $\widetilde{X}$ freely on the left by isometries.  Let $\mathcal{S}=(\widetilde{X}\times Y)/\sim$ where the identification is defined by the diagonal action: for all $\tilde{x}\in \widetilde{X}$, $y\in Y$, $[\gamma]\in \pi_1(X)$
\begin{equation}\label{e:identification}
(\tilde{x},y)\sim\bigl([\gamma^{-1}]\cdot \tilde{x},\varphi([\gamma])\cdot y\bigr).
\end{equation}
Let $\sigma:\widetilde{X}\times Y\rightarrow \mathcal{S}$ be the quotient map that identifies the orbits to points, and denote the points of $\mathcal{S}$ by $s_0=[(x_0,y_0)].$ The leaves of the associated foliation, $\mathcal{F}_{\mathcal{S}},$ are 
\begin{equation}
L_{s_0}= \{ [(\tilde{x},y_0)]\,|\, \tilde{x}\in \widetilde{X}\},
\end{equation}
and the leaf closures are
\begin{equation}
\overline{L}_{s_0}= \{ [(\tilde{x},y)]\,|\, \tilde{x}\in \widetilde{X}, y\in \mathcal{O}_{y_0}\},
\end{equation}
by virtue of the density of $G_0$ in $G$.  (See Lemma 2.1 of \cite{Ri2}.)

Note that this foliation depends on the choice of $X$ and the maximal tori; however, the transverse properties of the singular foliation defined by the leaf closures that are of interest in computing the $G$-invariant wave trace will be seen to be independent of these choices. For example, observe that the transverse structure of the leaf closures depends only on $G$ and not the choice of the subgroup $G_0.$

Let $T\mathcal{F}_{\mathcal{S}}$ denote the associated distribution and let $N\mathcal{F}_{\mathcal{S}}$ denote its complement. The dimension of $T\mathcal{F}_{\mathcal{S}}$ is the dimension of $\widetilde{X}$, which we denote by $p.$  Locally the metric on $\mathcal{S}$ is the product metric, and it is bundle-like for the foliation--i.e., the distance between leaves is locally constant with respect to this metric.\footnote{More precisely, a metric is bundle-like for a foliation $\mathcal{F}$ if for every open $U \subset M$ and for all vector fields $X$ and $Y$ that are perpendicular to the leaves and that satisfy $[X,Z]\in T\mathcal{F}$ and $[Y,Z]\in T\mathcal{F}$ for all $Z\in T\mathcal{F},$ then $g(X,Y)$ is a basic function for the foliation.} 
Via the work of K. Richardson in \cite{Ri2}, it is known that there is a metric space isometry between the space of leaf closures of the above foliation and the orbits of $Y$. As a consequence, no leaf of the associated foliation is dense, and the dimension of space of basic functions for the associated foliations is greater than zero, by virtue of the initial assumptions on the action.

It is known that there is a correspondence, which we will denote by $\Gamma,$ between the $G$-invariant functions on $Y$ and the functions on $S$ that are constant along the leaves (and hence the leaf closures, by continuity). Such functions are said to be {\it basic} with respect to the foliation, and we denote them by $C^\infty_B(\mathcal{S}, \mathcal{F}_\mathcal{S}).$ The correspondence $\Gamma$ is as follows:  a function $f$ on $Y$ that is constant along the orbits of $G$ may be regarded as a function on the product $\widetilde{X}\times Y,$ which is constant on the closures of the leaves, and thus can be identified with a function $F$ on the suspension $\mathcal{S}$, which correspond to functions on the product $\widetilde{X}\times Y$ that satisfies
\begin{equation}\label{e:functionalident}
F(\tilde{x},y)=F([\gamma^{-1}]\cdot \tilde{x},y\cdot\varphi([\gamma])).
\end{equation} 
Similarly, if a smooth function $F$ on $\mathcal{S}$ is basic for the foliation on $\mathcal{S}$ then, by definition, it is constant along the leaves of $\mathcal{S}.$  Thus, it can be identified with a function on $Y.$ Furthermore, $F$ must also be constant on the leaf closures of $\mathcal{S}$ by continuity.  Thus, such a function can be naturally identified with a function of $Y$ which is constant along the orbits of $G.$  

Given this identification, it is easily seen either directly (as below) or from the work of K. Richardson in \cite{Ri2}, that the following lemma holds:

\begin{lemma}\label{t:isospectrality}
Let $\Delta^{\mathcal{S}}$ denote the Laplacian in the above metric on $\mathcal{S}.$
The basic spectrum of $(S,\mathcal{F}_\mathcal{S})$ is identical to the $G$-invariant spectrum of $Y.$
\end{lemma}

\begin{proof}
First recall the mean curvature one-form of a foliation: 
\begin{equation}
\kappa(Z)=\sum_{i=1}^pg(\nabla_{E_i}E_i, Z),\quad \text{where }Z\in C^\infty(N\mathcal{F}_{\mathcal{S}}),
\end{equation}
where $g$ is the metric on $\mathcal{S}$ and $\{E_1,\,E_2,\dots,\,E_p\}$ is an orthonormal basis of $T\mathcal{F}_{\mathcal{S}}$. Via \cite{PaRi}, it is known that if the mean curvature form satisfies the condition that $\iota_X \kappa=\iota_X d\kappa=0$ for all vector fields $X$ tangent to the leaves, then the basic spectrum is contained in the spectrum of the ordinary Laplacian on functions . Hence, it is possible to define a basic projector, $P:C^\infty(M)\rightarrow C^\infty_B(\mathcal{S}, \mathcal{F}_\mathcal{S}).$ In this case, $\kappa(Z)$ is always zero since the metric $g$ is locally just the product metric $g_X+g_Y,$ where $g_X$ and $g_Y$ are the metrics on $X$ and $Y,$ respectively. Thus,
\begin{equation}\label{e:laps}
\Delta_B^\mathcal{S}P=P\Delta^\mathcal{S},
\end{equation}
where $\Delta_B^\mathcal{S}$ is the Laplacian on $\mathcal{S}$ restricted to the basic functions.
From the above discussion of $\Gamma,$ it is easily seen that
\begin{equation}
\Delta^\mathcal{S}_BP\Gamma=\Gamma\Delta^G,
\end{equation}
and the result follows immediately.
\end{proof}

Next, we relate the orbit structure of the action of $G$ on $Y$ to the structure of leaf closures of $\mathcal{S}.$ Recall, the orbit structure of a locally smooth group $G$ on the manifold $Y.$  (See, for example, Chapter IV of \cite{Bre}.) Let $H$ be a principal isotropy subgroup for the action, and let $Y_{(H)},$ denote the union of orbits of $G$ on $Y$ that are of orbit type $G/H.$  These form an open dense set, as noted earlier, where $d$ denotes the maximal orbit dimension.  There may be other orbits of the same dimension, the exceptional orbits, of orbit type $G/K$ where $H$ is conjugate to a subgroup of $K$ (we assume, without loss of generality, that $H\subset K$) and $K/H$ is a finite, non-trivial group. 
For any subgroup $K'$ of $G$, let $Y_{(K')}$ denotes the set of orbits of type $G/K'$ on $Y,$ and let $E_{(K')}$ be the exceptional orbits--those of the same dimension as those in $Y_{(K')},$ but not of the same orbit type:
\begin{equation*}
E_{(K')}=\{y\in \overline{Y_{(K')}}\,|\, dim (\mathcal{O}_y)= dim (G/K'),\,type (\mathcal{O}_y)\not= type(G/K')\}.
\end{equation*}
For each $j\le d,$ the union of the orbits of dimension $j$ is, therefore,
\begin{equation}
Y_j=(Y_{(K')}\cup E_{(K')}),
\end{equation}
where $K'$ is conjugate to a subgroup of all the isotropy groups of $Y_{(K')}$ with $dim(G/K')=j.$

Next recall the structure of the leaf closures on a foliated manifold, see for example, \cite{Molino}. The leaves of an arbitrary foliation are not generally closed; however, the closure of any leaf is an embedded submanifold of $M$ and is itself a union of leaves, and is foliated by the leaves that it contains. Further, the leaf closures generally vary in dimension, and are defined by a variable-dimensional, completely integrable distribution $T\overline{\mathcal{F}}$. The variation of the leaf closure dimension over the foliated manifold defines a natural stratification (Section 5.4, \cite{Molino}), as follows.  Let $\ell(x)$ be the function that assigns to a point $x$ the dimension of the leaf closure containing $x.$  This function takes its values in the positive integers $\{p+k\}$ where $p$ is the dimension of the leaves, and $k$ ranges over $0\le k_{1}\le k\le k_{N}< q,$ with $k_1$ and $k_N$ denoting the minimal and maximal values for $k$, respectively. It is known that the function $\ell(x)$ is lower semi-continuous \cite{Molino}, Chapter 5. Let $\Sigma_{p+k}$ denote the set $\{\ell^{-1} (p+k)\}.$

The correspondence between the orbit structure of $G$ on $Y$ and the stratification of the associated foliation $(\mathcal{S},\mathcal{F}_\mathcal{S})$ is given by the following.
\begin{lemma}\label{orbitstrat}
Let $p=dim(\widetilde{X})$ as above. For the associated foliation $(\mathcal{S},\mathcal{F}_\mathcal{S}),$ the strata of $(\mathcal{S},\mathcal{F}_\mathcal{S})$ correspond to the union of orbits of the same dimension: 
\begin{equation}\label{e:strat}
\Sigma_{p+k}=(\widetilde{X}\times Y_{k})/\sim.
\end{equation}

\end{lemma}
\begin{proof}
This follows purely from the characterization of the leaf closures above, and dimensionality considerations.  
%
%
\end{proof}


\begin{rem}
From the above, it follows from the above lemma that the wave invariants of the basic spectrum of the associated foliation are precisely those of the $G$-invariant spectrum. We will show that these invariants depend only on the action of $G$ on $Y,$ and not the subgroup $G_0,$ by virtue of the density of $G_0$ in $G.$

\end{rem}

\subsection{Summary of related foliation results}
\label{summary}

This section is organized analogously to Section \ref{mainresults}. We begin by defining the singular phase space for the space of leaf closures for a foliated manifold, and its associated stratification. Let $(M,\,\mathcal{F})$ be a manifold foliated by $p$-dimensional leaves, and let $\pi:T^*M\rightarrow M$ denote the usual map on the cotangent bundle $(T^*M,\,\omega),$ equipped with its usual symplectic form. Let $T\overline{\mathcal{F}}$ denote the subset of $T M$ consisting of vectors tangent to the leaf closures. This subset is a completely integrable distribution of variable dimension. The space
\begin{equation}\label{e:scs}
(T\overline{\mathcal{F}})^0=\{\xi_x\in T^*_xM\,|\,\xi_x(v_x)=0\,\,\forall v_x\in (T\overline{\mathcal{F}})_x\}.
\end{equation}
is the singular phase space for the space of leaf closures. As in the case of the obit space of a $G$-action, it is stratified, as follows:

%

\begin{defn}\label{t:stratdefn}
In the notation above,
\begin{equation}
(T\overline{\mathcal{F}})^0=\bigcup_{k=0}^{k_N}\Sigma^*_k
\end{equation}
where for $k>0$
\begin{equation}\label{e:sigmastark}
\Sigma^*_k=(N\Sigma_{p+k})^0\cap (T\overline{\mathcal{F}})^0,
\end{equation}
and 
\begin{equation}\label{e:sigmastarzero} 
\Sigma^*_0=\Bigl(\bigcup_{\ell} \bigl((N\Sigma^*_{p+\ell})^0\bigr)^c\cup(N\Sigma_p)^0\Bigr)\cap (T\overline{\mathcal{F}})^0,
\end{equation}
where the union is taken over $\ell$ with $0\le \ell<k_{N}.$ Later, we will wish to distinguish between certain exceptional points in $\Sigma^*_0$, thus let
\begin{equation}\label{e:sigmaexc}
\Sigma^*_e=\Bigl(\bigcup_{\ell} \bigl((N\Sigma^*_{p+\ell})^0\bigr)^c\Bigr)\cap (T\overline{\mathcal{F}})^0.
\end{equation}
Note that in \eqref{e:sigmastark}, when $k=k_{N},$ $\Sigma_{k_{N}}^*=(T\overline{\mathcal{F}})^0|_{\Sigma_{N}},$ which is open and dense in $(T\overline{\mathcal{F}})^0.$ 
\end{defn}


Note that the subset of vectors tangent to the leaves of the foliation, $T\mathcal{F},$ is contained in $T\overline{\mathcal{F}}$ and hence the analogously defined $(T\mathcal{F})^0$ contains our stratified phase space $(T\overline{\mathcal{F}})^0.$ This fact is significant because $(T\mathcal{F})^0$ is coisotropic, and as such, admits a foliation (the null-foliation) by $p$-dimensional leaves. Further, $(T\overline{\mathcal{F}})^0$ is a saturated subset of $(T\overline{\mathcal{F}})^0,$ and hence is also foliated by $p$-dimensional leaves, (see Section 3, \cite{San5}). Further, we can associate to $(T\overline{\mathcal{F}})^0$ a groupoid, $\mathcal{G}^*(\overline{\mathcal{F}}),$ with special properties, given in the subsequent proposition:  
    
\begin{prop}\label{t:groupoidprop}
There exists a groupoid $\mathcal{G}^*(\overline{\mathcal{F}})$ associated to $(T\overline{\mathcal{F}})^0$ such that the orbits of $\mathcal{G}^*(\overline{\mathcal{F}})$ are $(p+k)$-dimensional leaves that foliate each stratum $\Sigma^*_k.$ Furthermore, for each $k$ each of the $(p+k)$-dimensional leaves of $\Sigma^*_k$ is a union of $p$-dimensional leaves induced by the null foliation on $(T\mathcal{F})^0$. The stratum $\Sigma^*_0$ is foliated by $p$-dimensional leaves, which are orbits of the groupoid as well.
\end{prop} 

\begin{rem}\label{r:holonomygpoid} The structure of this groupoid will be necessary later, so we will elaborate a bit more with regard to its definition. This is a generalization of the holonomy groupoid with respect to the leaves on a foliated manifold $M,$  denoted by $\mathcal{G}(\mathcal{F}),$ whose definition we recall below from \cite{W}. 
Its elements, $\boldsymbol{\alpha},$ are ordered triples $\boldsymbol{\alpha}=\bigl[x,y,[\alpha] \bigr]$ where $x$ and $y$ are points belonging to the same leaf $L$ of $(M,\mathcal{F}),$ where $[\alpha]$ is an equivalence class of piecewise smooth curves lying entirely in $L$ with $x=\alpha(0)$ and $y=\alpha(1).$ This groupoid on $M$ lifts to a groupoid $\mathcal{G}^*(\mathcal{F})$ on $(T\mathcal{F})^0,$ whose elements define local diffeomorphisms $h_\alpha$ of local transversals in the usual fashion, and via the infinitesimal holonomy map, $dh_\alpha,$ define a holonomy action on certain transverse covectors as follows. The natural action of $\mathcal{G}(\mathcal{F})$ on $(T\mathcal{F})^0$ to is defined for $\xi_{x}\in (T_{x}\mathcal{F})^0$ by
\begin{equation}\label{e:holaction}
\forall X_{y}\in (T_y\mathcal{F})^0\quad (\boldsymbol{\alpha}\cdot\xi)_{y}(X_{y})=\xi_{x}(dh_\alpha^{-1}(X_{y})),
\end{equation}
where $dh_\alpha:(T_{x}\mathcal{F})^0\rightarrow (T_{y}\mathcal{F})^0$ is the differential of the holonomy map of the holonomy element $\boldsymbol{\alpha}.$ 
If $\eta_y=dh^*_{\alpha^{-1}}(\xi_x)$, then the triples $[\xi_x, \eta_y,dh^*_{\alpha^{-1}}]$ form the holonomy groupoid $\mathcal{G}^*(\mathcal{F})$ over $(T\mathcal{F})^0.$ As noted above, the null foliation of the space $(T\mathcal{F})^0$ has $p$-dimensional leaves which are orbits of $\mathcal{G}^*(\mathcal{F}).$ The generalization to $\mathcal{G}^*(\overline{\mathcal{F}})$ is achieved by considering triples $\boldsymbol{\overline{\alpha}}=\bigl[x,y,[\overline{\alpha}] \bigr]$ where $x$ and $y$ are points belonging to the same leaf  closure $\overline{L}$ of $(M,\mathcal{F}),$ where $[\overline{\alpha}]$ is an equivalence class of piecewise smooth curves lying entirely in $\overline{L}$ with $x=\alpha(0)$ and $y=\alpha(1).$ The rest of the construction of $\mathcal{G}^*(\overline{\mathcal{F}})$ is defined analogously.
Note that $\mathcal{G}^*(\mathcal{F})|_{(T\overline{\mathcal{F}})^0}\subset \mathcal{G}^*(\overline{\mathcal{F}}).$
\end{rem}


Next, we consider the hamiltonian flow and its behavior with respect to the singular phase space. Let $\Phi^t(\xi_x)$ denote the hamiltonian curve associated to $H(\xi_x)=|\xi_x|_x,$ which is determined by the bundle-like metric on $M$. We recall from \cite{San3} that  $\Phi^t(\xi_x)$ restricts to $(T\mathcal{F})^0,$ where $H(\xi_x)=H_{\mathcal{F}^\perp}(\xi_x),$ the part of the hamiltonian that arises from the transverse metric on $M.$ This transverse flow is particularly well-behaved with respect to the singular phase space $(T\overline{\mathcal{F}})^0$ and its stratification, due to the bundle-like nature of the metric on $M.$ In fact, on each stratum $\Sigma^*_k$ the hamiltonian function $H(\xi_x)=H_{\overline{\mathcal{F}}^\perp}(\xi_x),$ the part of the hamiltonian that arises from the metric that is transverse to the foliation by leaf closures. Further, the transverse flow sends leaves to leaves, with respect to all types of leaves in Proposition \ref{t:groupoidprop}.  Furthermore, we have the following proposition:

%
%
\begin{prop}\label{t:flowstrat}
The transverse hamiltonian flow of $H$ restricts to the singular phase space.  In other words, if $\xi_x\in (T\overline{\mathcal{F}})^0$ (with $\xi_x\in\Sigma^*_k$, say) then the hamiltonian vector field $\Xi_H(\xi_x)\in T\Sigma^*_k$ and 
\begin{equation}\label{e:flowrestricts}
\{\Phi^t(\xi_x)\,|\,t\in\mathbb{R}\}\subset (T\overline{\mathcal{F}})^0.
\end{equation} 
Furthermore,  if $k>0$ then either 
\begin{equation}\label{e:singstrat1}
\{\Phi^t(\xi_x)\,|\,t\in\mathbb{R}\}\subset \Sigma_k^*
\end{equation}
or
\begin{equation}\label{e:singstrat2}
\{\Phi^t(\xi_x)\,|\,t\in\mathbb{R}\}\subset \Sigma_k^*\cup\Sigma_e^*.
\end{equation}
Furthermore, the intersection of $\{\Phi^t(\xi_x)\,|\,t\in\mathbb{R}\}$ with $\Sigma_e^*$ is finite.
If $k=0$ and $\Sigma_p$ is non-empty and does not consist of isolated leaves, then if $\Phi^t(\xi_x)$ is tangent to $(N\Sigma_p)^0\cap(T\overline{\mathcal{F}})^0\subset \Sigma^*_0$ then 
\begin{equation}\label{e:singtrat3}
\{\Phi^t(\xi_x)\,|\,t\in\mathbb{R}\}\subset (N\Sigma_p)^0\cap(T\overline{\mathcal{F}})^0.
\end{equation}
Otherwise, $\Xi_H(\xi_x)$ is not tangent to $\Sigma^*_e$.
\end{prop}
It follows that a hamiltonian curve can be tangent to only one of the $\Sigma^*_k$.
%
%
%

Equipped with these results we can define a notion of a relatively closed curve with respect to the foliation:

\begin{defn}\label{d:relclosed}
An arc of the hamiltonian flow $\Phi^t$ through $\xi_x$ will be said to be {\it relatively closed with respect to the singular Riemannian foliation} $\overline{\mathcal{F}}$ with period $T$ if for endpoints $\xi_x$ and $\eta_y=\Phi^T(\xi_x)$ there is a homotopy class $[\overline{\alpha}]$ of a curve $\overline{\alpha}$ wholly contained in the leaf closure with endponts $x$ and $y$, such that $dh^*_{\overline{\alpha}^{-1}}(\xi_x)=\Phi^T(\xi_x).$
This is equivalent to the existence of a groupoid element of the form $[\xi_x,\Phi^T(\xi_x),dh^*_{\overline{\alpha}^{-1}}]\in \mathcal{G}^*(\overline{\mathcal{F}}).$ Let $\mathcal{RT}(M,\mathcal{F})$ denote the set of lengths of hamiltonian arcs that are relatively closed with respect to the foliation of $(M,\mathcal{F})$.
\end{defn}

\begin{rem}
We note that $\gamma(t,x)=\pi(\Phi^t(x,\xi))$ are geodesics in $M.$ Let $\Xi$ denote the hamiltonian vector field of the transverse bundle-like metric. Then, in local distinguished coordinates, it is easily seen that $\gamma'(0,x)=(d\pi)_{\xi_x}(\Xi_H)\perp T_x\overline{\mathcal{F}}.$ From Chapter 6 of \cite{Molino}, it is known that if $\gamma(t,x)$ is a geodesic passing through $x$, and is perpendicular to the leaf closures at one point, then this geodesic remains perpendicular to all the leaf closures that it meets. Thus, the projections of such relatively closed hamiltonian curves are geodesic arcs that are orthogonal to all the leaf closures through which the geodesic passes.

\end{rem}

From \cite{San3}, we have the following result:
\begin{thm}\label{t:sojourntimes}
In the notation previously established, 
\begin{equation}\label{e:invariants}
WF\bigl(Trace(U_B(t,x,y) \bigr)\subset\{(T,\tau)\,|\,\tau<0\,,\,T\in\mathcal{RT}(M,\mathcal{F})\}.
\end{equation}

\end{thm}

Let $Z^T$ denote the union of relatively closed hamiltonian arcs of period $T$. Using Proposition \ref{t:flowstrat}, we observe that each $Z^T$ can be decomposed with respect to the stratification as follows:

\begin{equation}
Z^T=\bigcup_{k=0}^{k_N} Z^T_k,
\end{equation}
where for each $k$, $Z^T_k$ is the union of all the relatively closed curves of length $T$ that are tangent to $\Sigma_k^*$. (In some cases, $Z^T_k$ may be empty.)

Let $\kappa(T)$ be the smallest integer such that $Z^T_{\kappa(T)}$ is non-empty.
 
\begin{defn}\label{d:clean}
Recalling the various foliations of the strata in Proposition \ref{t:groupoidprop}, let $\widetilde{T\mathcal{F}_k}$ denote the distribution defining the foliation of $\Sigma^*_k$, and let $\widetilde{N\mathcal{F}_k}$ denote the transverse distribution.  The holonomy $dh_{\overline{\alpha}}$ defines a holonomy action on the leaves of the foliation $(\Sigma^*_k,\widetilde{T\mathcal{F}_k}),$ which we denote by $d\widetilde{h}_{(\overline{\alpha},\xi_x)},$ as in \cite{San3}, Section 2.2. We say that the relatively closed set of curves $Z^T_k$ is {\it clean} if (1)$ Z^T_k$ is a smooth submanifold of $(T\mathcal{F})^0;$ and (2) for every $\xi_x\in Z^T_k$ with $\eta_y=(dh_{\overline{\alpha}}^{-1})^*\xi_x=\Phi^T(\xi_x)$ then $d\Phi^T_{\xi_x}(T_{\xi_x}Z^T_k)=T_{\eta_y}Z^T_k$ for $\overline{\alpha}\in \mathcal{G}(\overline{\mathcal{F}})$ with $\overline{\alpha}(0)=x$ and $\overline{\alpha}(1)=y.$ The condition that $\eta_y=(dh_{\overline{\alpha}}^{-1})^*\xi_x$ implies that for all $\xi_x\in Z^T_k$
\begin{equation}\label{e:cleanone}
d\Phi^T_{\xi_x}(\widetilde{N_{\xi_x}\mathcal{F}_k})=d\widetilde{h}_{(\overline{\alpha},\xi_x)}(\widetilde{N_{\xi_x}\mathcal{F}_k})=\widetilde{N_{\eta_y}\mathcal{F}_k}.
\end{equation}
\end{defn}

\begin{lemma}\label{l:density}
There exists a smooth canonical density, $d\mu_{Z^T_k},$ on each component of the relative fixed point set $Z^T_k.$
\end{lemma}

Let $\Gamma^T=\{(T,\tau)\,|\,\tau<0\}$ denote the ray over $T\in \mathcal{RT}.$ Henceforward, we assume that the set $Z^T_k$ of relative fixed points of the hamiltonian flow $\Phi^T$ on $(T\mathcal{F})^0$ is clean for all $T\in \mathcal{RT}(M,\mathcal{F})$ in the sense of Definition \ref{d:clean}. The relatively closed hamiltonian arcs of a given length $T$ make up conic submanifolds $Z^T_k$ whose connected components are finite in number and denoted by $Z^{T,\ell}_k.$ Let $S(Z^{T,\ell}_k)$ be the set $\{(T,\tau)\in Z^{T,\ell}_k\,|\,|\tau|=1\},$ let $e^{T,\ell}_k:=dim\bigl(S(Z^{T,\ell}_k)\bigr),$ let $e^T_k=max\{e^{T,\ell}_k\},$ and let $e_T=e^T_{\kappa(T)}.$ 

\begin{thm}\label{t:fulltrace}
Let $t=T\in \mathcal{RT}(M,\mathcal{F})$ and suppose that the $Z^T_k$ satisfy Definition \ref{d:clean}. Then, in the above notation, near $t=T,$
\begin{equation}\label{e:sumfoliations}
Trace(U_B(t,x,y)\bigr)=\sum_{T\in \mathcal{RT}(M,\mathcal{F})}\nu_{T}(t),
\end{equation}
where $\nu_T\in I^{-1/4-e_T/2-m}(\mathbb{R},\Gamma^T,\mathbb{R})$ where $m=-(p+\kappa(T))/2.$  Furthermore, $\nu_{T}$ has an expansion of the form
\begin{equation}\label{e:expansion}
\nu_{T}(t)=\sum_{j=0}^\infty\sigma_j(T)(t-T+i0)^{-\frac{e_T-1}{2}-\frac{p+\kappa(T)}{2}-j}\,mod\,C^\infty(\mathbb{R}),
\end{equation}
and modulo Maslov factors. In the above, the coefficients $\sigma_j(T)$ contain contributions from stationary phase arguments and involve integrals over $S(Z^T_{\kappa(T)})$, and may also have contributions involving the fixed point sets $Z^T_k$  for $k>\kappa(T)$ when such sets are non-empty.

The leading term $\sigma_0(T)$ is
\begin{equation}\label{e:symbol}
e^{\frac{i\pi m_T}{4}}\Bigl[\int_{S(Z^T_{\kappa(T)})}\sigma(U)\,d\mu_{Z^T_{\kappa(T)}}\Bigr]\,\tau^{\frac{e_T-1}{2}-\frac{p+\kappa(T)}{2}}\sqrt{d\tau},
\end{equation}
where $m_T$ is the Maslov index of $Z^{T}_{\kappa(T)},$ and $\sigma(U)$ is the symbol of the Schwartz kernel of the wave operator, $e^{-it\Delta},$ where $\Delta$ is the ordinary laplacian acting on functions on $M.$  
\end{thm}

\subsection{Proof of the main results}
\label{proofs}

Proving the main results is a matter of translating the foliations result on $\mathcal{S}$ to the $G$-manifold $Y$ by means of the unique structure of the suspension. Accordingly, this section is organized as follows. We begin by demonstrating the correspondence between the stratified structures of $(T\overline{\mathcal{F}}_\mathcal{S})^0$ and $(TG)^0,$ in the remark below. We then demonstrate the correspondence between the transverse flow $\Phi^t$ on $(T\overline{\mathcal{F}}_S)^0$ and the flow $\Phi^t_Y$ on $(TG)^0\subset T^*Y,$ showing that Proposition \ref{t:flowstrat} implies Proposition \ref{t:hamdyn}. We then prove the correspondence between Definitions \ref{d:relclosedyg} and \ref{d:relclosed}, and finally show how Theorems \ref{t:sojourntimes} and \ref{t:fulltrace} and corresponding results of the previous section imply Theorems \ref{t:sojourntimesgroup} and \ref{t:fulltracegroup} and related results for the $G$-invariant trace.

\begin{rem}\label{r:manifoldstructure}
In what follows, we will frequently make use of the special manifold structure of the suspension. The essential fact about the suspension $\mathcal{S}$ is that its manifold structure is the unique one such that for every simply connected open set $U\subset X$, there is a diffeomorphism $\phi_U: \sigma(\widetilde{U}\times Y)\rightarrow U\times Y,$ where $\widetilde{U}$ is the inverse image of $U$ via the projection $\tilde{p}:\widetilde{X}\rightarrow X.$ Thus, by considering an open cover of simply connected open sets $\{U_\alpha\}_{\alpha\in \mathcal{A}}$ of $X,$ we form an open cover of $\mathcal{S}$ defined by $\{\sigma(\widetilde{U}_\alpha\times Y)\}_{\alpha\in \mathcal{A}}$ with respect to which the manifold structure of $\mathcal{S}$ factors through the manifold structure of $X\times Y$ via the local diffeomorphisms $\phi_{\alpha}$. The manifold structure $T^*\mathcal{S}$ factors through the manifold structure of $T^*(X\times Y),$ in a similar fashion: let $\{W_{\alpha,\beta}\}$ be an open cover of $T^*\mathcal{S}$ such that each $W_{\alpha,\beta}=\pi^{-1}(\sigma(\tilde{p}^{-1}(U_\alpha)\times V_{\beta}))$ where $\{U_{\alpha}\}$ is as above, and
$\{V_{\beta}\}$ is an atlas of $Y.$ Then $\{W_{\alpha,\beta}\}_{\alpha,\beta}$ is an atlas of $T^*\mathcal{S},$ and on each chart let  $\tilde{\phi}_{\alpha,\beta}$ denote the corresponding coordinate map. Then  $\tilde{\phi}_{\alpha,\beta}=\phi'_{\alpha,\beta}\circ\phi_{\alpha, \beta}$ where $\phi_{\alpha, \beta}:W_{\alpha,\beta}\rightarrow T^*(X\times Y)$ and $\phi'_{\alpha,\beta}$ is a coordinate map on the open set $W'_{\alpha,\beta}:=\phi_{\alpha, \beta}(W_{\alpha,\beta})$. Then $\{W'_{\alpha,\beta}\}_{\alpha,\beta}$ form an atlas of $T^*(X\times Y).$
\end{rem}

The proof of Proposition \ref{t:hamdyn} requires the following two lemmas.

\begin{lemma}\label{l:correspondence}
Let $\pi_2: X\times Y \rightarrow Y$ be the usual projection. There exists a local diffeomorphism $\phi: T^*\mathcal{S}\rightarrow  T^*(X\times Y)$ such that 
\begin{equation}\label{e:phasecorr}
\phi\bigl((T\overline{\mathcal{F}}_\mathcal{S})^0\bigr)=d\pi_2^*\bigl( (TG)^0\bigr),
\end{equation}
and for each $k$
\begin{equation}\label{e:stratcorr}
\phi(\Sigma^*_k)=d\pi_2^*(Y^*_k).
\end{equation}
\end{lemma}

\begin{proof}
We have the following natural surjective mappings: $\tilde{\pi}_2: \widetilde{X}\times Y \rightarrow Y,$ $\sigma: \widetilde{X}\times Y\rightarrow \mathcal{S},$ and $\tilde{p}\times id_Y:\widetilde{X}\times Y\rightarrow X\times Y$ where $id_Y$ is the identity on $Y$. Note that $\tilde{\pi}_2= \pi_2\circ (\tilde{p}\times id_Y).$ Thus the following maps are injective: $d\tilde{\pi}_2^*:  T^*Y\rightarrow T^*(\widetilde{X}\times Y),$ $d\sigma^*: T^*\mathcal{S}\rightarrow T^*(\widetilde{X}\times Y)$ and $d(\tilde{p}\times id_Y)^*:T^*(X\times Y)\rightarrow T^*(\widetilde{X}\times Y),$ and $d\pi_2^*:T^*Y\rightarrow T^*(X\times Y)$ with $d\tilde{\pi}_2^*= d(\tilde{p}\times id_Y)^*\circ d\pi_2^*.$ We define the injective mapping $\phi$ as follows:
\begin{equation}
\phi=\bigl(d(\tilde{p}\times id_Y)^*\bigr)^{-1}\circ d\sigma^*.
 \end{equation}
With respect to an open cover of $T^*\mathcal{S}$ of the form given in Remark \ref{r:manifoldstructure}, it is a local diffeomorphism.

Next, we claim that 
\begin{equation}\label{e:pitilde}
d\sigma^*(\Sigma^*_k)=d\tilde{\pi}_2^*(Y^*_k),
\end{equation}
and hence \eqref{e:stratcorr} follows via the relation $d\tilde{\pi}_2^*= d(\tilde{p}\times id_Y)^*\circ d\pi_2^*.$
Equation \eqref{e:pitilde} is easily proved via a standard containment argument. First suppose $k>0$. Let $q=(\tilde{x},y)\in \tilde{X}\times Y,$ and let $s=\sigma((\tilde{x}, y))$ and consider $\tilde{\xi}_q\in d\sigma^*(\Sigma^*_k)$. We observe that $\tilde{\xi}_q\in d\sigma^*(\Sigma_k^*)$ if and only if $\tilde{\xi}_q= d\sigma^*(\xi_s)$ where $\xi_s\in (T\overline{\mathcal{F}}_\mathcal{S})^0\cap (N\Sigma_{p+k})^0$. Hence,
\begin{eqnarray*}
\xi_s(v_s)&=&0 \, \forall v_s\in (T\overline{\mathcal{F}}_\mathcal{S})\nonumber\\
\xi_s(v_s)&=&0 \, \forall v_s\in N\Sigma_{p+k}=N(\sigma(\tilde{X}\times Y_k))\nonumber\\
\end{eqnarray*}
But
\begin{eqnarray*}
v_s\in (T\overline{\mathcal{F}}_\mathcal{S})&\iff& v_s=d\sigma(\tilde{v}_q) \text{ for } \tilde{v}_q\in T_{\tilde{x}}\tilde{X}\oplus T_y\mathcal{O}_y\label{e:tfzero2}\\
v_s\in N(\sigma(\tilde{X}\times Y_k))&\iff& v_s=d\sigma(\tilde{v}_q) \text{ for } \tilde{v}_q\in N(\tilde{X}\times Y_k).\label{e:normal2}
\end{eqnarray*}
Note that $N(\tilde{X}\times Y_k)\subset T(\tilde{X}\times Y)$ is just $\{0\}\oplus NY_k$. Thus, together, the above conditions imply that $d\sigma^*(\xi_s)$ annihilates vectors with components in $T(\tilde{X})$, $T_y\mathcal{O}_y$, or $N_yY_k$, which defines $d\tilde{\pi}^*(Y^*_k)$. The reverse conclusion is similar.
For $k=0$, the proof is similar and only slightly more complicated. Equation \eqref{e:phasecorr} follows from \eqref{e:stratcorr}.
\end{proof}

\begin{lemma}\label{l:inducedflow}
Let $\Phi^t$ denote the restriction of the transverse flow to $(T\overline{\mathcal{F}}_\mathcal{S})^0$ as in Proposition \ref{t:flowstrat}, and let $\Phi^t_Y$  be the flow restricted to $(TG)^0.$ Let $s=\sigma\bigl((\tilde{x},y)\bigr)$ as in the previous lemma. Then for $\eta_{s}\in (T\overline{\mathcal{F}}_{\mathcal{S}})^0$ and $\eta_y\in (TG)^0,$ such that $d\sigma^*(\eta_{s})=d\pi_2^*(\eta_y)$ we have
\begin{equation}\label{e:flowphi}
\phi\bigl(\Phi^t(\eta_{s})\bigr)=d\pi_2^*\bigl(\Phi^t_Y(\eta_y)\bigr).
\end{equation}
\end{lemma}

\begin{proof}
As in the previous lemma, equation \eqref{e:flowphi} follows from the following claim: 
\begin{equation}\label{e:flow}
d\sigma^*\bigl(\Phi^t(\eta_{s})\bigr)=d\tilde{\pi}_2^*\bigl(\Phi^t_Y(\eta_y)\bigr).
\end{equation}
Recall from the paragraph preceding Proposition \ref{t:flowstrat} that $H_{\overline{\mathcal{F}}^\perp}(\eta_{s})$ is the hamiltonian function associated to the flow $\Phi^t$ on $(T\overline{\mathcal{F}}_\mathcal{S})^0$ that is transverse to the leaf closures. Because the metric on $\mathcal{S}$ is locally the product metric, it follows that for $\eta_{s}\in(T\overline{\mathcal{F}}_\mathcal{S})^0,$ and $\eta_y\in (TG)^0,$ such that $d\sigma^*(\eta_{s})=d\tilde{\pi}_2^*(\eta_y),$ one has

\begin{equation}
d\sigma^*\bigl(H_{\overline{\mathcal{F}}^\perp}(\eta_{s})\bigr)=d\tilde{\pi}_2^*(|\eta_y|_y).
\end{equation}
Note that the expressions appearing on both sides of the equality above have the appropriate invariance under the action of $G$. 
\end{proof}

The proof of Proposition \ref{t:hamdyn} now follows easily:
\begin{proof}
Proof of Proposition \ref{t:hamdyn}. From Lemma \ref{l:correspondence}, $\phi$ carries the stratification of $(T\overline{\mathcal{F}}_\mathcal{S})^0,$ to the corresponding stratification of $(TG)^0.$ Lemma \ref{l:inducedflow} implies that the flows are the same, and thus must have the same behavior with respect to the stratification. The result then follows from Proposition \ref{t:flowstrat}.
\end{proof}

In order to prove Theorems \ref{t:sojourntimesgroup} and \ref{t:fulltracegroup}, we need to first establish the equivalence of Definitions \ref{d:relclosedyg} and \ref{d:relclosed} with the following proposition.

\begin{prop}\label{p:equivalencedef}
Suppose $s_1,\,s_2\in \mathcal{S}$ such that $s_1=\sigma\bigl((\tilde{x}_1, y_1)\bigr)$ and $s_2=\sigma\bigl((\tilde{x}_2, y_2)\bigr)$ where $y_2\in \mathcal{O}_{y_1}$ and suppose further that $\phi(\xi_{s_1})=d\pi_2^*(\xi_{y_1}),$ and, similarly, 
$\phi(\eta_{s_2})=d\pi_2^*(\xi_{y_2})$. Then $\xi_{s_1}$ and $\eta_{s_2}$ are endpoints of a transverse hamiltonian arc of length $T$ that is relatively closed with respect to the foliation $\overline{\mathcal{F}}_\mathcal{S}$ if and only if $\xi_{y_1}$ and $\eta_{y_2}$ are endpoints of a corresponding hamiltonian arc of length $T$ that is relatively closed in $(TG)^0\subset T^*Y$ with respect to the action.
\end{prop}
\begin{proof}
The proof follows from Lemma \ref{l:inducedflow}, Definitions \ref{d:relclosedyg} and \ref{d:relclosed}, and the claim that the holonomy of the associated foliation and the holonomy of the leaf closures arises from the group action.

To justify this claim, first consider the holonomy of the leaves of the foliation of $\mathcal{S},$ rather than the leaf closures. 
Recalling the notation of Section 3.2 and Remark \ref{r:manifoldstructure}, suppose $\boldsymbol{\alpha}\in\mathcal{G}^*(\mathcal{F})|_{(T\overline{\mathcal{F}_\mathcal{S}})^0}$ is such that the corresponding homotopy class $[\alpha]$ can be represented by a curve that is contained in $\sigma(\tilde{p}^{-1}(U)\times Y)$, where $U$ is a simply connected subset of $X$. Then the holonomy action associated to $\boldsymbol{\alpha}$ is that of a simple foliation, and $\sigma(\tilde{p}^{-1}(U)\times Y)$ is diffeomorphic to $U\times Y\subset X\times Y.$ The holonomy action there is associated to the null-foliation of $d\pi_2^*\bigl((TG)^0\bigr)=(ker \,d\pi_2)^0\cap T^*X\times (TG)^0$ where
\begin{equation*}
(ker\, d\pi_2)^0=\{\xi\in T^*(X\times Y)\,|\, \xi_{(x,y)}(V_{(x,y)})=0\,\forall V_{(x,y)}\in (ker \,d\pi_2)\}.
\end{equation*}
This submanifold is also foliated by the canonical lift to $T^*(X\times Y)$ of $(ker\,d\pi_2)$, and the lifted foliation is simple. Thus, the holonomy action is that given by identifying vectors with the corresponding tangent vectors on the base. 
If $\boldsymbol{\alpha}$ is not of the type above, then it is the concatenation of curves as above and a curve that can be represented by some $[\gamma]\in \pi_1(X)$. Let $g=\phi([\gamma])$.  In this case, the corresponding holonomy action is that given by $dL^*_{g^{-1}}=:\widetilde{L}_g,$ $g\in G_0.$  

 
%
The holonomy of the leaf closures can similarly be realized by sliding along paths contained in the leaf closures, as we show in the following paragraph. The holonomy of elements $\boldsymbol{\overline{\alpha}}\in\mathcal{G}^*(\overline{\mathcal{F}})$ that can be represented by a curve that is contained in $\sigma\bigl(\tilde{p}^{-1}(U)\times Y\bigr)$ for simply connected $U$ as above have the holonomy of a simple foliation. The holonomy action of elements that are not of this type are the concatenation of curves of the previous types, and curves that travel across different leaves in the leaf closure.
Each leaf closure corresponds to an orbit $\mathcal{O}_{y_0}$ of $G$ on $Y$, with isotropy subgroup $K$, and the holonomy of a path $[\overline{\alpha}]$ with non-trivial holonomy corresponds to a group element which defines a nontrivial coset of $K/H.$ For leaf closures with non-trivial holonomy, the non-trivial holonomy action will be represented by $g\in G$ that represent non-trivial cosets of $G/K$ where, in the notation of Section 3.1 $K$ is the isotropy subgroup of the orbit corresponding to the leaf closure.
 
Consider a holonomy element $\boldsymbol{\overline{\alpha}}$ corresponding to a path that is not homotopic to a loop contained entirely in a single leaf (that case is already covered by ordinary leaf holonomy). Such a path can be decomposed into a finite concatenation of piecewise smooth paths $\boldsymbol{\overline{\alpha}^i}$ that are either wholly contained a leaf of the original foliation or are homotopic to paths that have endpoints in different leaves. For example, one may divide up the arc into pieces that are contained in sets of the form $\sigma(\widetilde{U}\times Y),$ as above.  We may assume that such paths can be locally decomposed into $(\overline{\alpha}_X^i,\overline{\alpha}_Y^i)$ where $\overline{\alpha}_X^i\subset X$ and $\overline{\alpha}_Y^i\subset Y.$ If the path is not homotopic to a loop in a leaf, then there must be at least two segments where the segment travels between leaves.  Let $(\overline{\alpha}_X^{i_0},\overline{\alpha}_Y^{i_0})$ be such a path.  We may further assume that $\overline{\alpha}_X^{i_0}$ is constant.  Then $\overline{\alpha}_Y^{i_0}(t_1)=h\cdot \overline{\alpha}_Y^{i_0}(t_0)$ for some $h\in G\setminus G_0$.  The holonomy map $h_{\boldsymbol{\overline{\alpha}^{i_0}}} $associated to this path segment is $\widetilde{L}_{h}$.  If the effect of travelling along the segments up to the $i_0$-th segment is $\widetilde{L}_{g'}$ for $g'\in G_0$, the the cumulative effect of travelling along the segments up through $i_0$ is $\widetilde{L}_{(g'h)}.$  The path may continue in the new leaf for some time, picking up some holonomy effects, before crossing back to the original leaf (or possibly on to other leaves), but the net effect of the holonomy will be the action by a product of elements in $G,$ not all of which will be in $G_0.$ Since the loop must return to the original leaf, the result of this product, say $g\in G$, determines the holonomy of the loop, will be some representative of $K/H,$ where $K$ is the isotropy subgroup of the orbit determined by the leaf closure, and $H$ is the principal isotropy group of the group action.

\end{proof}


%
\begin{proof}{Proof of Theorem \ref{t:sojourntimesgroup}.}

The result now follows easily from Proposition \ref{p:equivalencedef}, which implies that the relatively closed length spectrum of the foliation of $\mathcal{RT}(\mathcal{S},\,\mathcal{F}_\mathcal{S})$ is identical to the relatively closed length spectrum of the action of $G$ on $Y,$ $\mathcal{T}^G.$ The result then follows from Theorem \ref{t:sojourntimes}. \end{proof}
%

We now proceed with establishing the results leading to the proof of Theorem \ref{t:fulltracegroup}. First, we prove that the relative fixed point sets correspond to one another  and that the correspondence preserves the notions of clean-ness in Definitions \ref{d:cleangroup} and \ref{d:clean}. We then proceed with the proof of Theorem \ref{t:fulltracegroup}, which consists of showing that the basic wave invariants that appear in Theorem \ref{t:fulltrace} can be translated into invariants of the group action of $G$ on $Y.$


\begin{lemma}\label{t:fpsets}
In the notation previously established,
\begin{equation}
\phi(Z^T_k)=d\pi_2^*(F^T_k).
\end{equation}
Furthermore, the fixed point sets $F^T_k$ are clean in the sense of Definition \ref{d:cleangroup}  if and only if the sets $Z^T_k$ are clean in the sense of Definition \ref{d:clean}.
\end{lemma}
\begin{rem}\label{r:densesubgroup}
At first glance, it appears that $F^T_0$ depends on the holonomy given by $G_0$ rather than $G.$ However, this turns out not to be the case, due to the density of $G_0$ in $G.$ Let $F^T_0(H)$ denote the set of hamiltonian curves of length $T$ on $T^*Y$ which are relatively closed as above with respect to the action of a subgroup $H$ of $G.$ Then, $F^T_0(G)\subset F^T_0(G_0)$ since $G_0\subset G.$ The inclusion must in fact be an equality, since if we assume not, then there must be a $\xi_x\in F^T_0(G_0)$ that is not in $F^T_k(G).$ Thus, there exists some $g\in G\setminus G_0$ such that $dL_{g}^*\circ\Phi^T_Y(\xi_x)\not=\xi_x.$ However, since $G_0$ is dense in $G$, for any such $g\in G\setminus G_0$ it is possible to construct a sequence $\{g_k\}\subset G_0$ converging to $g$.  Since the composition $dL_{g}^*\circ\Phi^T_Y(\xi_x)$ is smooth, the sequence $\{dL_{g_k}^*\circ\Phi^T_Y(\xi_x)\}_k$ which is constant and equal to $\xi_x$ must converge to $dL_{g}^*\circ\Phi^T(\xi_x).$ Thus, $dL_{g}^*\circ\Phi^T_Y(\xi_x)=\xi_x$, and $F^T_0(G_0)=F^T_0(G).$ 

\end{rem}
\begin{proof}
The first part of the lemma follows from Lemma \ref{l:inducedflow} and the discussion of holonomy in Proposition \ref{p:equivalencedef}.
The second part of the lemma is due to the following observations concerning the action of holonomy, and from a comparision of Definitions \ref{d:clean} and Definition \ref{d:cleangroup}. Recall from Section 3.2 that the foliated submanifolds $(\Sigma^*_k,\widetilde{T{\mathcal{F}}_k})$ each have holonomy actions themselves.  First, recall that the left $G$-action lifts to $T^*\mathcal{S}$ in the usual way:  Let $s_i=\sigma\bigl((\tilde{x}_i,y_i) \bigr)$ for $i=1,\,2$ be as in Proposition \ref{p:equivalencedef}. Then $\forall g\in G$ we define the left action on $T^*\mathcal{S}$ by $\widetilde{L}_g(\xi_{s_1}):=dL_{g^{-1}}^*(\xi_{s_1}).$  For $k=0$ the holonomy induced by a loop $\alpha$ at $s_1=\sigma\bigl((\tilde{x}_1,y_1)\bigr)$ on $\widetilde{\mathcal{F}}$, denoted by $d\widetilde{h}_{(\alpha,\xi_{s_1})}$ is $d\widetilde{L}_g,$ for $g\in G_0.$  For any arc which is homotopic to one that is contained in $\sigma\bigl(\tilde{p}^{-1}(U)\times Y\bigr)$ for simply connected $U,$ the induced holonomy is just that given by identifying a vector in the tangent space at $\xi_{s_1}$ with the same vector in the tangent space at $\eta_{s_2}$. In a similar fashion, for $k>0$, the holonomy is given by $d\widetilde{L}_g,$ for $g\in G$, for loops, and via identifications along the fibre for curves with endpoints in the same plaque contained in $\sigma(\widetilde{U}\times Y)$ for simply connected $\tilde{p}(\widetilde{U})$.
For a suspension, the previous discussion of holonomy implies that equation \eqref{e:cleanone} can be reformulated as
\begin{equation}
d\Phi^T_{\xi_x}(\widetilde{N_{\xi_x}\mathcal{F}}_{k}|_{Z^T_k})=d\widetilde{L}_g(\widetilde{N_{\xi_x}\mathcal{F}}_{k}|_{Z^{T}_k})\quad g\in G \end{equation}
when the holonomy elements are non-trivial. The result then follows by comparing the two definitions.
\end{proof}

Next, we note that Lemma \ref{l:densitygroup} is an immediate corollary to the following proposition that will also be used subsequently in the proof of Theorem \ref{t:fulltracegroup}.

\begin{prop}\label{t:diffeointegral}
If $f$ is a function defined on $S(Z^T_k)$ then there exists densities $d\tilde{\mu}^T_k$ on $S(F^T_k)$ such that 
\begin{equation}\label{e:fixedk}
\int_{S(Z^T_k)} f\,d\mu_{Z^T_k}=\int_{S(F^T_k)}(\pi_2)_*\bigl((\phi^{-1})^*f)\bigr)\,d\tilde{\mu}^T_k,
\end{equation}
where $\pi_2:T^*(X\times Y)\rightarrow T^*Y.$ 
\end{prop}

\begin{proof} 
To understand how the density $d\tilde{\mu}^T_k$ is defined, first recall from the proof of Lemma \ref{l:density}, whose proof appears in \cite{San5}, the fact that the sets $Z^T_k$  (and, thus, the sets $S(Z^T_k)$ also) are equipped with canonical densities which are arise by combining a transverse density with a leaf-wise density, as follows. Let $T(Z^T_k)=\widetilde{T\mathcal{F}_k}\oplus \widetilde{N\mathcal{F}_k}$ denote the natural splitting of the tangent space of $T(Z^T_k).$ For any foliation of leaf dimension $\ell,$ we construct a leaf-wise density by recalling the characteristic form of a foliation $\chi_{\mathcal{F}_\ell}:$ let $X_1, \dots, X_\ell$ be in $T\mathcal{F}_\ell$, and let $\{E_j\}_{j=1}^n$ be an orthonormal frame of $\mathcal{S},$ of dimension $n.$  Then one defines a canonical $\ell$-form via the metric on $\mathcal{S}$ as follows:
\begin{equation}
\chi_{\mathcal{F}_\ell}(X_1,\dots,X_\ell)=det\bigl(g_{ij}(E_i,X_j)\bigr).
\end{equation}
Using this, we can define a canonical leaf-wise density on $\widetilde{T\mathcal{F}_\ell}$ by lifting it to $(T\mathcal{F}_\ell)^0,$ via $\pi:T^*\mathcal{S}\rightarrow \mathcal{S}.$ Thus, $|d\pi^*(\chi_{\mathcal{F}_\ell})|$ defines a positive leafwise density on  $(T\mathcal{F}_\ell)^0.$ Now, for $Z^T_k$, there are two foliations--the $p$-dimensional foliation, and the $(p+k)$-dimensional foliation associated to the leaf closures. Each of these foliations gives rise to two such densities. Let $|d\pi^*(\chi_{\mathcal{F}})|$ denote the $p$-dimensional leaf-wise density, and let $|d\pi^*(\chi_{\overline{\mathcal{F}}})|$ denote the $(p+k)$-dimensional density. Next, since the horizontal space, $\widetilde{N\mathcal{F}_k}$ of $((T\overline{\mathcal{F}_{\mathcal{S}}})^0|_{Z^T_k},\,\widetilde{T\mathcal{F}}_k)$ is a symplectic space, and $\Phi^T$ and $d\widetilde{h}_{\boldsymbol{\overline{\alpha}}}$ are symplectic diffeomorphisms of $\widetilde{N\mathcal{F}_k},$ one can use section 4 of \cite{DG} to construct canonical densities on $\widetilde{N\mathcal{F}_k}.$ Let $d\mu'_{Z^T_k}$ denote the transverse density arising from the $(p+k)$-dimensional foliation. By definition of $(T\overline{\mathcal{F}_\mathcal{S}})^0,$ this transverse density depends only on the action of $G$ on $Y.$ 
One then constructs the full density on $Z^T_k$ as follows:
\begin{equation}\label{e:densityonZ}
d\mu_{Z^T_k}=d\pi^*(\chi_{\overline{\mathcal{F}}})\otimes d\mu'_{Z^T_k}.
\end{equation}
Here we use the fact that $|\widetilde{T\mathcal{F}_k}|\otimes|\widetilde{N\mathcal{F}_k}|\cong|T(Z^T_k)|.$ We can further decompose the leaf-wise measure for the $(p+k)$-dimensional leaf closures of $Z^T_k$ as follows:
\begin{equation}
d\mu_{Z^T_k}=d\pi^*(\chi_{\mathcal{F}})\otimes d\nu^T_k\otimes d\mu'_{Z^T_k}.
\end{equation}
where
\begin{equation}
|d\nu^T_k|=|d\pi^*(\chi_{\mathcal{F}})|^{-1}\otimes |d\pi^*(\chi_{\overline{\mathcal{F}}})|.
\end{equation}
%
%

Now consider the left-hand-side of \eqref{e:fixedk}. Essentially, we proceed by localizing this integral, making a change of variables via the local diffeomorphism to the product, and integrate over the first factor in $T^*X\times T^*Y$ (which is isomorphic to $T^*(X\times Y)$). The density $d\tilde{\mu}^T_k$ arises from pushing forward via $\pi_2$ the density $d\nu^T_k\otimes d\mu'_{Z^T_k}.$

More precisely, the integral in \eqref{e:fixedk} can be localized via a partition of unity subordinate to an atlas of the form described in Remark \ref{r:manifoldstructure} localized to $S(Z^T_k)$ as follows. As in that remark, let $\{W_{\alpha,\beta}\}_{\alpha,\beta}$ denote such an atlas of $S(Z^T_k)$, and let $\{\psi_{\alpha,\beta}\}_{\alpha,\beta}$ denote a partition of unity on $S(Z^T_k)$ that is subordinate to that atlas. Furthermore, observe that $\phi$ restricted to each $W_{\alpha,\beta}$ gives us the $\phi_{\alpha, \beta}$ of Remark \ref{r:manifoldstructure}. Then, 
\begin{equation}
\int_{S(Z^T_k)} f\,d\mu_{Z^T_k}=\sum_{\alpha,\beta}\psi_{\alpha,\beta}\int_{W_{\alpha, \beta}} f \, d\mu_{Z^T_k}.
\end{equation}
Each integral on the right-hand-side of the above can be rewritten via the local diffeomorphism $\phi$ as follows:
\begin{equation}
\int_{W_{\alpha, \beta}} f\, d\mu_{Z^T_k}=\int_{\phi(W_{\alpha, \beta})}\bigl((\phi^{-1}_{\alpha, \beta})^* f\bigr)\, (\phi^{-1}_{\alpha, \beta})^*(d\mu_{Z^T_k})
\end{equation}
Let $(s,\zeta)$ be local coordinates on $SZ^T_k$ such that $\phi_{\alpha, \beta}(s,\zeta)=(x,y,\xi, \eta)$ are local coordinates on $\phi_{\alpha, \beta}(W_{\alpha, \beta}).$ Note that by Lemma \ref{t:fpsets}, $\{\phi_{\alpha, \beta}(W_{\alpha, \beta})\}_{\alpha,\beta}$ restricts to an atlas of $d\pi_2^*(SF^T_k)$. Thus, making a change of coordinates the right side of \eqref{e:fixedk}, can we written as
\begin{equation}\label{e:localization}
\sum_{\alpha,\beta}\psi_{\alpha,\beta}\circ\phi^{-1}_{\alpha,\beta}\int_{\phi_{\alpha, \beta}(W_{\alpha, \beta})}\bigl((\phi^{-1}_{\alpha, \beta})^* f\bigr)\,(\phi^{-1}_{\alpha,\beta})^*(d\mu_{Z^T_k}).
\end{equation}
Note that $\{\psi_{\alpha,\beta}\circ\phi^{-1}_{\alpha,\beta}\}$ is a partition of unity on $d\pi_2^*(SF^T_k)$  subordinate to the atlas $\{\phi_{\alpha, \beta}(W_{\alpha, \beta})\}_{\alpha,\beta}$. 

Consider next the local densities in \eqref{e:localization}. From \eqref{e:densityonZ}, we have 
\begin{equation}
d\mu_{Z^T_k}=d\pi^*(\chi_\mathcal{F})\otimes d\nu^T_k\otimes d\mu'_{Z^T_k}.
\end{equation}
In the local coordinates defined above, $\chi_{\mathcal{F}_\mathcal{S}}=d\pi^*d\pi_1^*dvol(X).$  Thus, under the local diffeomorphisms $\phi_{\alpha,\beta}$ 
\begin{equation}\label{e:locdens}
(\phi^{-1}_{\alpha,\beta})^*(d\mu_{Z^T_k})=d\pi^*d\pi_1^*dvol(X)\otimes d(\tilde{\nu}^T_k)_{\alpha,\beta}.
\end{equation}
In the above, $(d\tilde{\nu}^T_k)_{\alpha,\beta}$ is a local density defined by the diffeomorphisms $\phi_{\alpha,\beta}$, which can be regarded as a density on the conormal bundle to each orbit arising from $d\nu^T_k\otimes d\mu'_{Z^T_k}.$
Summing up over first $\alpha$ and then $\beta$ in the integral in \eqref{e:localization}, we get a global density on $d\pi_2^*(S(F^T_k))$ of the form, $d\pi^*d\pi_1^*dvol(X)\otimes d\tilde{\nu}^T_k.$

Putting \eqref{e:locdens} into \eqref{e:localization}, and integrating over $T^*X$ gives the result via the fact that $vol(X)=1$. This yields an integral over $S(F^T_k)$ with density defined by the local expressions above; we denote this density by $d\tilde{\mu}^T_k,$ which we note is just the push-forward of $d\tilde{\nu}^T_k$ by $\pi_2.$ With this observation the integrand is as given in \eqref{e:fixedk}.

\end{proof}

\begin{proof}{Proof of Theorem \ref{t:fulltracegroup}. }
The method of proof is a generalization of the methods used to prove the partial wave trace formula in \cite{San3}.
First note that the hypotheses of Theorem \ref{t:fulltrace} which require the clean-ness of the relatively closed curves of the associated foliation are equivalent to the clean-ness hypothesis of Theorem \ref{t:fulltracegroup}, by Lemma \ref{t:fpsets}. Then, from Theorem \ref{t:fulltrace}, and \cite{Ri2}, Lemma \ref{t:isospectrality}, the $G$-invariant spectrum of $Y$ admits, under the given hypotheses, a wave trace formula of the form \eqref{e:pwt1}, where each $\nu_T\in I^k$ with 
\begin{equation}\label{e:order}
k=-1/4-\frac{e_T}{2}+(p+\kappa(T))/2. 
\end{equation}
Specializing to the foliation on the suspension $\mathcal{S}$ yields
\begin{equation}\label{e:order2}
k=-1/4-\frac{max_j\, dim\, S(Z^T_{\kappa(T),j})}{2}+\frac{dim(\widetilde{X})+\kappa(T)}{2}. 
\end{equation}
Here $max_j\,\,dim (Z^T_{\kappa(T),j})$ is the dimension of the largest dimensional component of $S(Z^T_{\kappa(T)}),$ over the finite number of such components, and $p+\kappa(T)$ is the dimension of the leaf closures in the stratum of the associated foliation (with $p=dim(\widetilde{X})$).

Observe from the above that the codimension of the foliation of $(S(Z^T_k),\,\widetilde{\mathcal{F}})$ is equal to the dimension of $S(F^T_k)=q^T_k.$ But $S(Z^T_{\kappa(T)})$ is foliated by the $p$-dimensional leaves of $\widetilde{\mathcal{F}}$, and thus $max\,\,dim (S(Z^T_{\kappa(T)}))=p+q^T_{\kappa(T)},$ by above observation. Furthermore, for the associated foliation, $\kappa(T)=d(T).$  Thus, the $p$ terms cancel, and \eqref{e:pwt1} follows immediately from \eqref{e:sumfoliations}.  

It remains to show that the the coefficients in the expansion \eqref{e:expansion} can be expressed in terms of the $G$-manifold $Y$. But this follows from Proposition \ref{t:diffeointegral}, which allows the translation of the coefficients of the trace formula in 
Theorem \ref{t:fulltrace}, and the fact that $X$ has unit volume. In particular, consider the leading order term in the expansion. To avoid confusion, let $\sigma(U_\mathcal{S})$ denote the symbol of the Schwartz kernel of the wave operator, $U_\mathcal{S},$ acting on functions on $\mathcal{S}$ whose canonical relation is denoted $\Lambda^\mathcal{S},$ and similarly let 
$\sigma(U_Y)$ denote the symbol of the Schwartz kernel of the wave operator, $U_Y,$ acting on functions of $Y,$ with corresponding canonical relation $\Lambda^Y.$ Recall from the proof of the ordinary trace formula in \cite{DG} that for the Laplacian (which has subprincipal symbol equal to zero), the symbol of the Schwartz kernel of $U_\mathcal{S}$ is associated to the half density
\begin{equation}
pr^*_\mathcal{S}\bigl( |dt|^{1/2} \otimes |dz_1\wedge d\zeta_1|^{1/2}\bigr)
\end{equation}
where $$pr_\mathcal{S}:\Lambda^\mathcal{S}\rightarrow T^*(\mathbb{R}\times\mathcal{S})\setminus \{0\}$$ is the projection given by
$pr_\mathcal{S}(t,\tau,s_1,\zeta_1,s_2,\zeta_2)=(t,s_1,\zeta_1).$
But via the local diffeomorphism $\phi$ induces a local diffeomorphism $\tilde{\phi}: T^*(\mathbb{R}\times\mathcal{S})\rightarrow T^*(X\times Y\times \mathbb{R})$ in the obvious way. Under $\tilde{\phi}$ the density $ |dt|^{1/2} \otimes |dz_1\wedge d\zeta_1|^{1/2}$ becomes, 
\begin{equation}
|dt|^{1/2}\otimes |dx_1\wedge dy_1\wedge d\xi_1\wedge d\eta_1|^{1/2}
\end{equation}
where $(x_1,y_1,\xi_1,\eta_1)$ are coordinates on $T^*(X\times Y).$ (Here, we are tacitly making use of the isomorphism between $T^*(X\times Y)$ and $T^*X\times T^*Y.$) When this function is restricted to $S(Z^T_k)$ and pushed forward and via $(\pi_2)_*$, it becomes the half density associated to $\sigma(U_Y):$
\begin{equation}
pr_Y^*\big(|dt |^{1/2}\otimes |dy_1\wedge d\eta_1|^{1/2}\bigr),
\end{equation}
where $pr_Y:\Lambda^Y\rightarrow T^*(\mathbb{R}\times Y)\setminus \{0\}$ is the projection given by
$$pr_Y(t,\tau,y_1,\eta_1,y_2,\eta_2)=(t,y_1,\eta_1),$$
proving \eqref{e:symbolgroup} and completing the proof.

\end{proof}

\end{document}